\documentclass{article}
\usepackage{amsmath}
\usepackage{amsfonts}
\usepackage{amssymb}
\usepackage{graphicx}
\usepackage{xy} \xyoption{all}
\newtheorem{theorem}{Theorem}[section]

\newtheorem{corollary}[theorem]{Corollary}

\newtheorem{example}[theorem]{Example}

\newtheorem{lemma}[theorem]{Lemma}

\newtheorem{proposition}[theorem]{Proposition}
\newtheorem{remark}[theorem]{Remark}

\newenvironment{proof}[1][Proof]{\noindent\textbf{#1.} }{\ \rule{0.5em}{0.5em}}
\begin{document}

\title{On intersections of two-sided ideals of Leavitt path
algebras}
\author{Songul Esin$^{(1)}$, Muge Kanuni$^{(2)}$ and Kulumani M. Rangaswamy$^{(3)}$\\$^{(1)}$ Tuccarbasi sok. Kase apt. No: 10A/25 Erenkoy \\Kadikoy, Istanbul, Turkey.\\E-mail: songulesin@gmail.com\\$^{(2)\text{ }}$Department of Mathematics, Duzce University, \\Konuralp Duzce, Turkey.\\E-mail: mugekanuni@duzce.edu.tr\\$^{(3)}$Department of Mathematics, University of Colorado,\\Colorado Springs, CO. 80919, USA. \\E-mail: krangasw@uccs.edu}
\date{}
\maketitle

\begin{abstract}
Let $E$ be an arbitrary directed graph and let $L$ be the Leavitt path algebra
of the graph $E$ over a field $K$. It is shown that every ideal of $L$ is an
intersection of primitive/prime ideals in $L$ if and only if the graph $E$
satisfies Condition (K). Uniqueness theorems in representing an ideal of $L$
as an irredundant intersection and also as an irredundant product of finitely
many prime ideals are established. Leavitt path algebras containing only
finitely many prime ideals and those in which every ideal is prime are
described. Powers of a single ideal $I$ are considered and it is shown that
the intersection ${\displaystyle\bigcap\limits_{n=1}^{\infty}}I^{n}$ 
is the largest graded ideal of $L$ contained in $I$. This leads to an
analogue of Krull's theorem for Leavitt path algebras.

\end{abstract}

\section{Introduction and Preliminaries}

Leavitt path algebras $L_{K}(E)$ of directed graphs $E$ over a field $K$ are
algebraic analogues of graph $C^{\ast}$-algebras $C^{\ast}(E)$ and have
recently been actively investigated in a series of papers (see e.g. \cite{AA},
\cite{AAPS}, \cite{ABCR}, \cite{R-1}, \cite{T}). These investigations showed,
in a number of cases, how an algebraic property of $L_{K}(E)$ and the
corresponding analytical property of $C^{\ast}(E)$ are both implied by the
same graphical property of $E$, though the techniques of proofs are often
different. The initial investigation of special types of ideals such as the
graded ideals, the corresponding quotient algebras and the prime ideals of a
Leavitt path algebra was essentially inspired by the analogous investigation
done for graph $C^{\ast}$-algebras. But an extensive investigation of the
ideal theory of Leavitt path algebras, as has been done for commutative rings,
is yet to happen.

This paper may be considered as a small step in exploring the multiplicative
ideal theory of Leavitt path algebras and was triggered by a question raised
by Professor Astrid an Huef. In the theory of $C^{\ast}$-algebras and, in
particular, that of the graph $C^{\ast}$-algebras, every (closed) ideal $I$ of
$C^{\ast}(E)$ is the intersection of all the primitive/prime ideals containing
$I$ (Theorem 2.9.7, \cite{D}). In a recent 2015 CIMPA research school in
Turkey on Leavitt path algebras and graph $C^{\ast}$-algebras, Professor an
Huef raised the question\ whether the preceding statement is true for ideals
of Leavitt path algebras. We first construct examples showing that this
property does not hold in general for Leavitt path algebras. We then prove
that, for a given graph $E$, every ideal of the Leavitt path algebra
$L_{K}(E)$ is an intersection of primitive/prime ideals if and only if the
graph $E$ satisfies Condition (K). A uniqueness theorem is proved in
representing an ideal of $L_{K}(E)$ as the irredundant intersection of
finitely many prime ideals. As a corollary, we show that every ideal of
$L_{K}(E)$ is a prime ideal if and only if (i) Condition (K) holds in $E$,
(ii) for each hereditary saturated subset $H$ of vertices, $|B_{H}|\leq1$ and
$E^{0}\backslash H$ is downward directed and (iii) the admissible pairs
$(H,S)$ (see definition below) form a chain under a defined partial order.
Equivalently, all the ideals of $L_{K}(E)$ are graded and form a chain under
set inclusion. Following this, Leavitt path algebras possessing finitely many
prime ideals are described. We also give conditions under which every ideal of
a Leavitt path algebra is an intersection of maximal ideals.

The graded ideals of a Leavitt path algebra possess many interesting
properties. Using these, we examine the uniqueness of factorizing a graded
ideal as a product of prime ideals. A perhaps interesting result is that if
$I$ is a graded ideal and $I=P_{1}\cdot\cdot\cdot P_{n}$ is a factorization of
$I$ as an irredundant product of prime ideals $P_{i}$, then necessarily all
the ideals $P_{i}$ must be graded ideals and moreover, $I=P_{1}\cap\ldots\cap
P_{n}$. We also prove a weaker version of this result for non-graded ideals.
Finally, powers of an ideal in $L_{K}(E)$ are studied. While $I^{2}=I$ for any
graded ideal $I$, it is shown that, for a non-graded ideal $I$ of $L_{K}(E)$,
its powers $I^{n}$ $(n\geq1)$ are all non-graded and distinct, but the
intersection of the powers ${\displaystyle\bigcap\limits_{n=1}^{\infty}}I^{n}$ 
is always a graded ideal and is indeed the largest graded ideal of
$L_{K}(E)$ contained in $I$. As a corollary, we obtain an analogue of Krull's
theorem (Theorem 12, section 7, \cite{ZS}) for Leavitt path algebras: For an
ideal $I$ of $L_{K}(E)$, the intersection 
${\displaystyle\bigcap\limits_{n=1}^{\infty}}I^{n}=0$ 
if and only if $I$ contains no vertices.

\textbf{Preliminaries}: For the general notation, terminology and results in
Leavitt path algebras, we refer to \cite{AA}, \cite{R-1} and \cite{T}. For
basic results in associative rings and modules, we refer to \cite{AF}. We give
below a short outline of some of the needed basic concepts and results.

A (directed) graph $E=(E^{0},E^{1},r,s)$ consists of two sets $E^{0}$ and
$E^{1}$ together with maps $r,s:E^{1}\rightarrow E^{0}$. The elements of
$E^{0}$ are called \textit{vertices} and the elements of $E^{1}$
\textit{edges}.

A vertex $v$ is called a \textit{sink} if it emits no edges and a vertex $v$
is called a \textit{regular} \textit{vertex} if it emits a non-empty finite
set of edges. An \textit{infinite emitter} is a vertex which emits infinitely
many edges. For each $e\in E^{1}$, we call $e^{\ast}$ a ghost edge. We let
$r(e^{\ast})$ denote $s(e)$, and we let $s(e^{\ast})$ denote $r(e)$. A\textit{
path} $\mu$ of length $|\mu|=n>0$ is a finite sequence of edges $\mu
=e_{1}e_{2}\cdot\cdot\cdot e_{n}$ with $r(e_{i})=s(e_{i+1})$ for all
$i=1,\cdot\cdot\cdot,n-1$. In this case $\mu^{\ast}=e_{n}^{\ast}\cdot
\cdot\cdot e_{2}^{\ast}e_{1}^{\ast}$ is the corresponding ghost path. A vertex
is considered a path of length $0$. The set of all vertices on the path $\mu$
is denoted by $\mu^{0}$.

A path $\mu$ $=e_{1}\dots e_{n}$ in $E$ is \textit{closed} if $r(e_{n}%
)=s(e_{1})$, in which case $\mu$ is said to be \textit{based at the vertex
}$s(e_{1})$. A closed path $\mu$ as above is called \textit{simple} provided
it does not pass through its base more than once, i.e., $s(e_{i})\neq
s(e_{1})$ for all $i=2,...,n$. The closed path $\mu$ is called a
\textit{cycle} if it does not pass through any of its vertices twice, that is,
if $s(e_{i})\neq s(e_{j})$ for every $i\neq j$.

An \textit{exit }for a path $\mu=e_{1}\dots e_{n}$ is an edge $e$ such that
$s(e)=s(e_{i})$ for some $i$ and $e\neq e_{i}$. We say the graph $E$ satisfies
\textit{Condition (L) }if every cycle in $E$ has an exit. The graph $E$ is
said to satisfy \textit{Condition (K)} if every vertex which is the base of a
closed path $c$ is also a base of another closed path $c^{\prime}$ different
from $c$.

If there is a path from vertex $u$ to a vertex $v$, we write $u\geq v$. A
subset $D$ of vertices is said to be \textit{downward directed }\ if for any
$u,v\in D$, there exists a $w\in D$ such that $u\geq w$ and $v\geq w$. A
subset $H$ of $E^{0}$ is called \textit{hereditary} if, whenever $v\in H$ and
$w\in E^{0}$ satisfy $v\geq w$, then $w\in H$. A hereditary set is
\textit{saturated} if, for any regular vertex $v$, $r(s^{-1}(v))\subseteq H$
implies $v\in H$.

Given an arbitrary graph $E$ and a field $K$, the \textit{Leavitt path algebra
}$L_{K}(E)$ is defined to be the $K$-algebra generated by a set $\{v:v\in
E^{0}\}$ of pair-wise orthogonal idempotents together with a set of variables
$\{e,e^{\ast}:e\in E^{1}\}$ which satisfy the following conditions:

\begin{enumerate}
\item[(1)] $s(e)e=e=er(e)$ for all $e\in E^{1}$.

\item[(2)] $r(e)e^{\ast}=e^{\ast}=e^{\ast}s(e)$\ for all $e\in E^{1}$.

\item[(3)] (The "CK-1 relations") For all $e,f\in E^{1}$, $e^{\ast}e=r(e)$ and
$e^{\ast}f=0$ if $e\neq f$.

\item[(4)] (The "CK-2 relations") For every regular vertex $v\in E^{0}$,
\[
v=\sum_{e\in E^{1},s(e)=v}ee^{\ast}.
\]

\end{enumerate}

Every Leavitt path algebra $L_{K}(E)$ is a $\mathbb{Z}$-graded algebra 
$L_{K}(E)={\displaystyle\bigoplus\limits_{n\in\mathbb{Z}}}L_{n}$ 
induced by defining, for all $v\in E^{0}$ and $e\in E^{1}$, $\deg
(v)=0$, $\deg(e)=1$, $\deg(e^{\ast})=-1$. Further, for each 
$n\in \mathbb{Z}$, the homogeneous component $L_{n}$ is given by
\[L_{n}=\left\{{\textstyle\sum}
k_{i}\alpha_{i}\beta_{i}^{\ast}\in L:\text{ }|\alpha_{i}|-|\beta
_{i}|=n\right\}  .
\]
An ideal $I$ of $L_{K}(E)$ is said to be a graded ideal if $I=$ 
${\displaystyle\bigoplus\limits_{n\in\mathbb{Z}}}(I\cap L_{n})$.

We shall be using the following concepts and results from \cite{T}. A
\textit{breaking vertex }of a hereditary saturated subset $H$ is an infinite
emitter $w\in E^{0}\backslash H$ with the property that $0<|s^{-1}(w)\cap
r^{-1}(E^{0}\backslash H)|<\infty$. The set of all breaking vertices of $H$ is
denoted by $B_{H}$. For any $v\in B_{H}$, $v^{H}$ denotes the element
$v-\sum_{s(e)=v,r(e)\notin H}ee^{\ast}$. Given a hereditary saturated subset
$H$ and a subset $S\subseteq B_{H}$, $(H,S)$ is called an \textit{admissible
pair.} The set $\mathbf{H}$ of all admissible pairs becomes a lattice under a
partial order $\leq^{\prime}$ under which $(H_{1},S_{1})\leq^{\prime}
(H_{2},S_{2})$ if $H_{1}\subseteq H_{2}$ and $S_{1}\subseteq H_{2}\cup S_{2}$.
Given an admissible pair $(H,S)$, the ideal generated by $H\cup\{v^{H}:v\in
S\}$ is denoted by $I(H,S)$. It was shown in \cite{T} that the graded ideals
of $L_{K}(E)$ are precisely the ideals of the form $I(H,S)$ for some
admissible pair $(H,S)$. Moreover, $L_{K}(E)/I(H,S)\cong L_{K}(E\backslash
(H,S))$. Here $E\backslash(H,S)$ is the \textit{Quotient graph of }$E$ in
which\textit{ }$(E\backslash(H,S))^{0}=(E^{0}\backslash H)\cup\{v^{\prime
}:v\in B_{H}\backslash S\}$ and $(E\backslash(H,S))^{1}=\{e\in E^{1}
:r(e)\notin H\}\cup\{e^{\prime}:e\in E^{1},r(e)\in B_{H}\backslash S\}$ and
$r,s$ are extended to $(E\backslash(H,S))^{0}$ by setting $s(e^{\prime})=s(e)$
and $r(e^{\prime})=r(e)^{\prime}$. For a description of non-graded ideals of
$L_{K}(E)$, see \cite{R-2}.

A useful observation is that every element $a$ of $L_{K}(E)$ can be written as
$a={\textstyle\sum\limits_{i=1}^{n}}k_{i}\alpha_{i}\beta_{i}^{\ast}$, 
where $k_{i}\in K$, $\alpha_{i},\beta_{i}$
are paths in $E$ and $n$ is a suitable integer. Moreover, 
$L_{K}(E)={\textstyle\bigoplus\limits_{v\in E^{0}}}
L_{K}(E)v={\textstyle\bigoplus\limits_{v\in E^{0}}}vL_{K}(E).$ 
Further, the Jacobson radical of $L_{K}(E)$ is always zero (see
\cite{AA}). Another useful fact is that if $p^{\ast}q\neq0$, where $p,q$ are
paths, then either $p=qr$ or $q=ps$ where $r,s$ are suitable paths in $E$.

Let $\Lambda$ be an arbitrary infinite index set. For any ring $R$, we denote
by $M_{\Lambda}(R)$ the ring of matrices over $R$ with identity whose entries
are indexed by $\Lambda\times\Lambda$ and whose entries, except for possibly a
finite number, are all zero. It follows from the works in \cite{A}, \cite{AM}
that $M_{\Lambda}(R)$ are Morita equivalent to $R$.

Throughout this paper, $E$ will denote an arbitrary graph (with no restriction
on the number of vertices and the number of edges emitted by each vertex) and
$K$ will denote an arbitrary field. For convenience in notation, we will
denote, most of the times, the Leavitt path algebra $L_{K}(E)$ by $L$.

\section{Intersections of prime ideals}

In this section, we give necessary and sufficient conditions under which every
ideal of a Leavitt path algebra $L$ of an arbitrary graph $E$ is the
intersection of prime/primitive ideals. As applications, conditions on the
graph $E$ are obtained under which (a) every ideal of $L$ is a prime ideal and
(b) when $L$ contains only a finite number of prime ideals. A uniqueness
theorem for irredundant intersections of prime ideals is also obtained. We
also obtain conditions under which every ideal of $L$ is an intersection of
maximal ideals.

Remark: In this and the next section, by an ideal $I$ we mean an ideal $I$ of
$L$ such that $I\neq L$.

\begin{lemma}
\label{Graded => Primitive intersection} \textit{Let }$I$\textit{ be a graded
ideal a Leavitt path algebra }$L$ of an arbitrary graph $E$\textit{. Then }%
$I$\textit{ is the intersection of all primitive (and hence prime) ideals
containing }$I$\textit{.}
\end{lemma}

\begin{proof}
Let $H=I\cap E^{0}$ and $S=\{v\in B_{H}:v^{H}\in I\}$. By \cite{T}, the graded
ideal $I=I(H,S)$, the ideal generated by $H\cup\{v^{H}:v\in S\}$. Also,
$L/I\cong L_{K}(E\backslash(H,S))$. Since the Jacobson radical of
$L_{K}(E\backslash(H,S))$ is zero, we conclude that the intersection of all
primitive ideals of $L_{K}(E\backslash(H,S))$ is $0$. This means that $I$ is
the intersection of all the primitive ideals of $L$ containing $I$. Moreover,
since every primitive ideal is prime, we conclude that $I$ is the intersection
of all prime ideals containing $I$.
\end{proof}

The next example shows that a non-graded ideal of a Leavitt path algebra need
not be an intersection of all the prime/primitive ideals containing it.

\begin{example}
\label{Laurent => No Prime Intersection} \textit{Let }$E$ be a graph with one
vertex $v$ and a loop $c$ so $s(c)=v=r(c)$. Thus $E$ is the graph
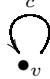
\[ \xymatrix{ {\bullet}_v \ar@(ur,ul)_c }\] 
\textit{Consider the ideal }$B=\left\langle p(c)\right\rangle $ of $L_{K}(E)$
where $p(x)$ is \textit{an irreducible polynomial in }$K[x,x^{-1}]$. We claim
that $B^{2}$ is not an intersection of prime ideals in $L_{K}(E)$. To see
this, first observe that $L_{K}(E)\overset{\theta}{\cong}R=K[x,x^{-1}]$ under
the map $\theta$ mapping $v\mapsto1,c\mapsto x$ and $c^{\ast}\mapsto x^{-1}$.
\textit{Then }$B\cong A=<p(x)>$\textit{, the ideal generated by }$p(x)$
in\textit{ }$R$. So it is enough if we show that\textit{ }$N=A^{2}$\textit{
cannot be an intersection of prime ideals of }$R$\textit{. }Suppose, on the
contrary, $N={\displaystyle\bigcap\limits_{\lambda\in\Lambda}}
M_{\lambda}$ where $\Lambda$ is an arbitrary index set and each $M_{\lambda}$
is a (non-zero) prime ideal of $R$. Note that each $M_{\lambda}$ is a maximal
ideal of $R$, as $R$ is a principal ideal domain. Then there is a homomorphism
$\phi:R\rightarrow{\displaystyle\prod\limits_{\lambda\in\Lambda}}
R/M_{\lambda}$ given by $r\longmapsto(\cdot\cdot\cdot,r+M_{\lambda},\cdot
\cdot\cdot)$ with $\ker(\phi)=N$. Now $\bar{A}=\phi(A)\cong A/N\neq0$
satisfies $(\bar{A})^{2}=0$ and this is impossible since 
${\displaystyle\prod\limits_{\lambda\in\Lambda}}R/M_{\lambda}$, 
being a direct product of fields, does not contain any
non-zero nilpotent elements.
\end{example}

We wish to point out that, unlike the case of graded ideals, there are
non-graded ideals in a Leavitt path algebra some of which (such as the ideal
$N$ in the above example) are not intersections of prime/primitive ideals and
there are also some non-graded ideals that are intersections of prime ideals.
To obtain an example of the latter, different from the above, consider the
Toeplitz algebra $L_{K}(E)$ where $E$ is a graph with two vertices $v,w$, a
loop $c$ with $s(c)=v=r(c)$ and an edge $f$ with $s(f)=v$ and $r(f)=w$. Thus
$E$ is the graph
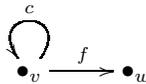
\[  \xymatrix{ {\bullet}_v \ar@(ur,ul)_c \ar@{->}[r]^f & {\bullet}_w }   \]
Then the ideals $A=\left\langle w,v+c\right\rangle $, $B=\left\langle
w,v+c^{2}\right\rangle $ and $I=\left\langle w,(v+c)(v+c^{2})\right\rangle $
are all non-graded ideals of $L_{K}(E)$ (see Proposition 6, \cite{R-2}). Now
$L_{K}(E)/\left\langle w\right\rangle \cong K[x,x^{-1}]$ under the map
$v+\left\langle w\right\rangle \longmapsto1,c+\left\langle w\right\rangle
\longmapsto x$ and $c^{\ast}+\left\langle w\right\rangle \longmapsto x^{-1}$.
Since $1+x$ and $1+x^{2}$ are irreducible polynomials in the principal ideal
domain $K[x,x^{-1}]$, $\left\langle 1+x\right\rangle \cap\left\langle
1+x^{2}\right\rangle =\left\langle (1+x)(1+x^{2})\right\rangle $ and moreover
the ideals $\left\langle 1+x\right\rangle $ and $\left\langle 1+x^{2}%
\right\rangle $ are maximal ideals. Thus $A/\left\langle w\right\rangle $ and
$B/\left\langle w\right\rangle $ are maximal ideals whose intersection is
$I/\left\langle w\right\rangle $. Hence the non-graded ideal $I=A\cap B$ is an
intersection of two primitive/prime ideals of $L_{K}(E)$.

We next explore conditions under which every ideal in a Leavitt path algebra
is an intersection of prime/primitive ideals.

To begin with, we need a result on lattice isomorphisms. In general, a lattice
isomorphism between two lattices need not preserve infinite infemums. But for
complete lattices, the infemum is preserved. This assertion is perhaps
folklore and we need it in the proofs of couple of statements below. Since we
could not find this statement explicitly stated or proved in our literature
search, we record it in the next Lemma and outline its easy proof.

\begin{lemma}
\label{Lattice Iso} \textit{Let }$f:(\mathbf{L,\leq)}\longrightarrow
(\mathbf{L}^{\prime},\leq)$\textit{ be an isomorphism of two complete
lattices. Let }$Y$ be \textit{any finite or infinite index set and let }
$P_{i}\in\mathbf{L}$ for each $i\in Y$. \textit{Then }
$f({\displaystyle\bigwedge\limits_{i\in Y}}P_{i})=
{\displaystyle\bigwedge\limits_{i\in Y}}f(P_{i})$\textit{.}
\end{lemma}

\begin{proof}
Let $A={\displaystyle\bigwedge\limits_{i\in Y}}P_{i}$. Clearly 
$f(A)\leq{\displaystyle\bigwedge\limits_{i\in Y}}f(P_{i})$. 
Suppose $B\leq f(P_{i})$ for all $i\in Y$. Then $f^{-1}(B)\leq
f^{-1}f(P_{i})=P_{i}$ for all $i$. Hence $f^{-1}(B)\leq
{\displaystyle\bigwedge\limits_{i\in Y}}P_{i}=A$. Then $B=f(f^{-1}(B))\leq f(A)$. 
Consequently, $f(A)={\displaystyle\bigwedge\limits_{i\in Y}}f(P_{i})$.
\end{proof}

\begin{lemma}
\label{No L => No prime intersection} Suppose $E$ is an arbitrary graph which
does not satisfy Condition (L). Then there is an ideal $I$ of the
corresponding Leavitt path algebra $L:=L_{K}(E)$ which is not an intersection
of prime/primitive ideals of $L$.
\end{lemma}

\begin{proof}
Since Condition (L) does not hold, there is a cycle $c$ based at a vertex $v$
having no exits in $E$. Now the ideal $A$ generated by the vertices on $c$ is
isomorphic to $M_{\Lambda}(K[x,x^{-1}])$ where $\Lambda$ is an index set
representing the set of all paths that end at $c$ but do not include all the
edges $e$ (see \cite{AAPS}). Also $M_{\Lambda}(K[x,x^{-1}])$ is Morita
equivalent to $K[x,x^{-1}]$ (see \cite{A}, \cite{AM}) and so its lattice of
ideals is isomorphic to the lattice of ideals of $K[x,x^{-1}]$ and that prime
(primitive) ideals correspond to prime (primitive) ideals under this
isomorphism (see \cite{AF}). \ Then, in view of Lemma \ref{Lattice Iso}, the
ideal $\overline{N}$ of $M_{\Lambda}(K[x,x^{-1}])$ that corresponds to the
ideal $N$ of $K[x,x^{-1}]$ constructed in Example
\ref{Laurent => No Prime Intersection} is not an intersection of prime ideals
and hence not an intersection of primitive ideals containing it. Let $I$
denote the ideal of $A$ that corresponds to the ideal $\overline{N}$ under the
isomorphism $A\longrightarrow M_{\Lambda}(K[x,x^{-1}])$. Now $M_{\Lambda
}(K[x,x^{-1}])$ \ and hence $A$ is a ring with local units and so every ideal
of $A$ is also an ideal of $L$. The existence of local units in $A$ also
implies that if $P$ is a prime ideal of $L$, then $P\cap A$ is a prime ideal
of $A$. Consequently, $I$ will be an ideal of $L$ which is not an intersection
of prime ideals of $L$. This implies that $I$ is also not an intersection of
primitive ideals.
\end{proof}

The next theorem describes conditions under which every ideal of a Leavitt
path algebra is an intersection of prime ideals.

\begin{theorem}
\label{Prime Intersection <=> Condition K} Let $E$ be an arbitrary graph. Then
the following\ properties are equivalent for $L:=L_{K}(E)$:
\end{theorem}

\begin{enumerate}
\item[(i)] \textit{Every ideal }$I$\textit{ of }$L$\textit{ is the
intersection of all the primitive ideals containing }$I$\textit{;}

\item[(ii)] \textit{Every ideal }$I$\textit{ of }$L$\textit{ is the
intersection of all the prime ideals containing }$I$\textit{;}

\item[(iii)] \textit{The graph }$E$\textit{ satisfies Condition (K).}
\end{enumerate}

\begin{proof}
Now (i)$\Rightarrow$(ii), since every primitive ideal is also a prime ideal.

Assume (ii). Suppose, on the contrary, $E$ does not satisfy Condition (K).
Then, by (Proposition 6.12, \cite{T}), there is an admissible pair $(H,S)$
where $H$ is a hereditary saturated subset of vertices and $S\subseteq B_{H}$
such that the quotient graph $E\backslash(H,S)$ does not satisfy Condition
(L). By Lemma \ref{No L => No prime intersection}, there is an ideal $\bar{I}$
in $L_{K}(E\backslash(H,S))$ which is not an intersection of all prime ideals
containing $\bar{I}$ in $L_{K}(E\backslash(H,S))$. Now $L/I(H,S)\cong
L_{K}(E\backslash(H,S))$ and so the ideal $A/I(H,S)$ that corresponds to
$\bar{I}$ under the preceding isomorphism is not an intersection of prime
ideals in $L/I(H,S)$. This implies that the ideal $A$ is not an intersection
of all the prime ideals containing $A$ in $L$. This contradiction shows that
$E$ must satisfy Condition (K), thus proving (iii).

Assume (iii) so that $E$ satisfies Condition (K). Then, by (Theorem 6.16,
\cite{T}), every ideal of $L$ is graded. By Lemma
\ref{Graded => Primitive intersection}, every ideal of $L$ is then an
intersection of primitive ideals. This proves (i).
\end{proof}

We next prove the uniqueness of representing an ideal of $L$ as an irredundant
intersection of finitely many prime ideals. We use the known ideas of proving
such statements. Recall that an intersection $P_{1}\cap\cdot\cdot\cdot\cap
P_{m}$ of ideals is \textit{irredundant} if no $P_{i}$ contains the
intersection of the other $m-1$ ideals $P_{j}$, $j\neq i$.

\begin{proposition}
\label{Uniqueness} \textit{Suppose }$A=P_{1}\cap\ldots\cap P_{m}=Q_{1}
\cap\ldots\cap Q_{n}$\textit{ are two representations of an ideal }$A$\textit{
of }$L$\textit{ as irredundant intersections of finitely many prime ideals
}$P_{i}$\textit{ and }$Q_{j}$\textit{ of }$L$\textit{. Then }$m=n$\textit{ and
}$\{P_{1},\ldots,P_{m}\}=\{Q_{1},\ldots,Q_{m}\}$\textit{.}
\end{proposition}

\begin{proof}
Now the product $Q_{1}Q_{2}\cdot\cdot\cdot Q_{n}\subseteq A\subseteq P_{1}$
and $P_{1}$ is a prime ideal. Hence $Q_{j_{1}}\subseteq P_{1}$ for some
$j_{1}\in\{1,\ldots,n\}$. Similarly, the product $P_{1}\cdot\cdot\cdot
P_{m}\subseteq A\subseteq Q_{j_{1}}$ and since $Q_{j_{1}}$ is prime,
$P_{i}\subseteq Q_{j_{1}}$ for some $i\in\{1,\ldots,m\}$. Thus $P_{i}\subseteq
P_{1}$ and, by irredundancy, $i=1$. Hence $P_{1}=Q_{j_{1}}$. Starting with
$P_{2}$, using similar arguments, we obtain $P_{2}=Q_{j_{2}}$ for some
$Q_{j_{2}}$. Now $Q_{j_{2}}\neq Q_{j_{1}}$, since otherwise $P_{1}=P_{2}$
which is not possible by irredundancy. Proceeding like this we reach the
conclusion that $\{P_{1},\ldots,P_{m}\}\subseteq\{Q_{1},\ldots,Q_{n}\}$.
Reversing the role of $P_{i}$ and $Q_{j}$, starting with the $Q_{j}$ and
proceeding as before, we can conclude that $\{Q_{1},\ldots,Q_{n}
\}\subseteq\{P_{1},\ldots,P_{m}\}$. Thus $m=n$ and $\{P_{1},\ldots
,P_{m}\}=\{Q_{1},\ldots,Q_{m}\}$\textit{.}
\end{proof}

We next explore the conditions on the graph $E$ under which will every ideal
of $L_{K}(E)$ is a prime ideal.

\begin{proposition}
\label{Everything Prime}Let $E$ be an arbitrary graph. Then the following are
equivalent for $L:=L_{K}(E)$:
\end{proposition}

\begin{enumerate}
\item[(a)] \textit{Every ideal of }$L$\textit{ is a prime ideal;}

\item[(b)] \textit{The graph }$E$\textit{ satisfies Condition (K), and }

\begin{enumerate}
\item[(i)] \textit{the set }$(\mathbf{H},\leq^{\prime})$\textit{ of all the
admissible pairs }$(H,S)$\textit{ in }$E$\textit{ is a chain under the defined
partial order }$\leq^{\prime}$\textit{, }

\item[(ii)] \textit{for each hereditary saturated set }$H$\textit{ of
vertices, }$|B_{H}|\leq1$\textit{ and }

\item[(iii)] \textit{for each }$(H,S)\in\mathbf{H}$\textit{, }$(E\backslash
(H,S))^{0}$\textit{ is downward directed;}
\end{enumerate}

\item[(c)] \textit{All the ideals of }$L$\textit{ are graded and form a chain
under set inclusion.}
\end{enumerate}

\begin{proof}
Assume (a). By Theorem \ref{Prime Intersection <=> Condition K}, the graph $E$
satisfies Condition (K). Suppose there are two admissible pairs $(H_{1}
,S_{1})$ and $(H_{2},S_{2})$ such that
\[
(H_{1},S_{1})\nleqslant^{\prime}(H_{2},S_{2})\text{ and }(H_{2},S_{2}
)\nleqslant^{\prime}(H_{1},S_{1}).
\]
Then the ideal $Q=I(\overline{H},\overline{S})$, where $(\overline
{H},\overline{S})=(H_{1},S_{1})\wedge(H_{2},S_{2})$, is not a prime ideal. To
see this, observe that the lattice $(\mathbf{H,\leq}^{\prime})$ of all the
admissible pairs is isomorphic to the lattice of graded ideals of $L$ (see
Theorem 5.7, \cite{T}) and that
\[
I(H_{1},S_{1})\cdot I(H_{2},S_{2})\subseteq I(H_{1},S_{1})\cap I(H_{2}
,S_{2})=I(\overline{H},\overline{S})
\]
but $I(H_{1},S_{1})\nsubseteq I(\overline{H},\overline{S})$ and $I(H_{2}
,S_{2})\nsubseteq I(\overline{H},\overline{S})$. Thus the set $\mathbf{H}$ of
all the admissible pairs must form a chain under the defined partial order
$\leq^{\prime}$. Now for any given hereditary saturated set $H$ of vertices,
$|B_{H}|\leq1$. Because, otherwise, $B_{H}$ will contain two subsets $S_{1}$
and $S_{2}$ such that $S_{1}\nsubseteq S_{2}$ and $S_{2}\nsubseteq S_{1}$ and
this will give rise to admissible pairs $(H,S_{1})\nleqslant^{\prime}
(H,S_{2})$ and $(H,S_{2})\nleqslant^{\prime}(H,S_{1})$, a contradiction. Also,
since for each $(H,S)\in\mathbf{H}$, $I(H,S)$ is a prime ideal, it follows
from Theorem 3.12 of \cite{R-1}, that $(E\backslash(H,S))^{0}$ is downward
directed. This proves (b).

Assume (b). Since the graph satisfies Condition (K), every ideal of $L$ is
graded and so is of the form $I(H,S)$ for some admissible pair $(H,S)\in
\mathbf{H}$. Since $(\mathbf{H,\leq}^{\prime})$ is a chain and is isomorphic
to the lattice of ideals of $L$, it follows that the ideals of $L$ form a
chain under set inclusion. This proves (c).

Assume (c). Let $P$ be any ideal of $L$. Suppose $I,J$ are two ideals such
that $IJ\subseteq P$. By hypothesis, one of them is contained in the other,
say $I\subseteq J$. Moreover, since the ideals are graded, $I=I^{2}$ ( by
Corollary 2.5, \cite{ABCR}) and so $I=IJ\subseteq P$. Thus $P$ is a prime
ideal and this proves (a).
\end{proof}

The following example illustrates the conditions of Proposition
\ref{Everything Prime}.

\begin{example}
\label{Example-everyone prime}Let $E$ be a graph with $E^{0}=\{v_{i}%
:i=1,2,\cdot\cdot\cdot\}$. For each $i$, there is an edge $e_{i}$ with
$r(e_{i})=v_{i}$ and $s(e_{i})=v_{i+1}$ and at each $v_{i}$ there are two
loops $f_{i},g_{i}$ so that $v_{i}=s(f_{i})=r(f_{i})=s(g_{i})=r(g_{i})$. Thus
$E$ is the graph
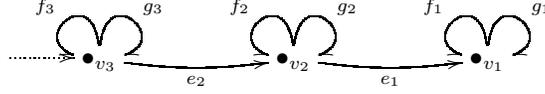
\[ 
\xymatrix{ \ar@{.>}[r] &
\bullet_{v_3}\ar@(u,l)_{f_3} \ar@(u,r)^{g_3} \ar@/_.3pc/[rr]_{e_2} & &  
\bullet_{v_2}\ar@(u,l)_{f_2} \ar@(u,r)^{g_2} \ar@/_.3pc/[rr]_{e_1}  && \bullet_{v_1}\ar@(u,l)_{f_1} \ar@(u,r)^{g_1} }
\]
Clearly $E$ is a row-finite graph and the non-empty proper hereditary
saturated subsets of vertices in $E$ are the sets $H_{n}=\{v_{1},\cdot
\cdot\cdot,v_{n}\}$ for some $n\geq1$ and form a chain under set inclusion.
Clearly, $E^{0}\backslash(H_{n},\emptyset)$ is downward directed for each $n$.
Thus the ideals of $L$ are graded prime ideals of the form $I(H_{n}%
,\emptyset)$ and they form a chain under set inclusion. It may be worth noting
that $L_{K}(E)$ does not contain maximal ideals.
\end{example}

\begin{remark}
\label{Primes always exist} As a property that is diametrically opposite of
the property of $L$ stated in Proposition \ref{Everything Prime}, one may ask
under what conditions a Leavitt path algebra contains no prime ideals. It may
be some interest to note that, while a Leavitt path algebra $L$ may not
contain a maximal ideal as indicated in Example \ref{Example-everyone prime},
$L$ will always contain a prime ideal. Indeed if the graph $E$ satisfies
Condition (K), then Theorem \ref{Prime Intersection <=> Condition K} implies
that $L$ contains prime ideals. Suppose $E$ does not satisfy Condition (K).
Then there will be a closed path $c$ based at a vertex $v$ in $E$ such that no
vertex on $c$ is the base of another closed path in $E$. If $H=\{u\in
E^{0}:u\ngeqq v\}$, then $E^{0}\backslash H$ is downward directed and, by
Theorem 3.12 of \cite{R-1}, $I(H,B_{H})$ will be a prime ideal of $L$. Thus a
Leavitt path algebra $L$ always contains a prime ideal.
\end{remark}

Another consequence of Theorem \ref{Prime Intersection <=> Condition K} is the
following Proposition.

\begin{proposition}
\label{Finitely many primes} Let $E$ be an arbitrary graph. Then the following
properties are equivalent for $L:=L_{K}(E)$:
\end{proposition}

\begin{enumerate}
\item[(a)] $L$\textit{ contains at most finitely many prime ideals;}

\item[(b)] $L$\textit{ contains at most finitely many prime ideals and all of
them are graded ideals;}

\item[(c)] \textit{The graph }$E$\textit{ satisfies Condition (K) and there
are only finitely many hereditary saturated subsets }$H$\textit{ of vertices
with the corresponding set of breaking vertices }$B_{H}$\textit{ finite;}

\item[(d)] $L$\textit{ has at most finitely many ideals.}
\end{enumerate}

\begin{proof}
Assume (a). If $L$ contains a non-graded prime ideal $P$ and $H=P\cap E^{0}$,
then, by Theorem 3.12 of \cite{R-1}, $P=\left\langle I(H,B_{H}
),f(c)\right\rangle $, where $f(x)$ is an irreducible polynomial belonging to
the Laurent polynomial ring $K[x,x^{-1}]$. Then, for each of the infinitely
many prime ideals $g(x)\in K[x,x^{-1}]$, we will have a prime ideal
$P_{g}=\left\langle I(H,B_{H}),g(c)\right\rangle $ of $L$, a contradiction. So
all the prime ideals of $L$ are graded. This proves (b).

Assume (b). Since every prime ideal of $L$ is graded, $E$ satisfies Condition
(K), by Corollary 3.13 of \cite{R-1}. Then, by Theorem
\ref{Prime Intersection <=> Condition K}, every ideal of $L$ is an
intersection of prime ideals. Since there are only finitely many distinct
possible intersections of finitely many prime ideals, $L$ contains only a
finite number of ideals all of which are graded. The conclusion of (c) then
follows from Corollary 12 of \cite{R-2}.

It is clear (c)$\Rightarrow$(d)$\Rightarrow$(a).
\end{proof}

\begin{remark}
\label{Finite intersection of primes} As a natural follow-up of Theorem
\ref{Prime Intersection <=> Condition K}, one may wish to explore the
conditions on $E$ under which every ideal of $L$, instead of being an
intersection of possibly infinitely many prime ideals, just an intersection of
no more than a finite number of prime ideals of $L$. In this case, $\{0\}$
will be the intersection of only a finite number of prime ideals. This means
$L$ itself contains only a finitely many prime ideals. Since every ideal of
$L$ is an intersection of prime ideals belonging to this finite set, it is
clear that $L$ must then contain only a finite number of ideals. Thus, this
property is then equivalent to the graph $E$ satisfying the condition of (c)
of the above Proposition \ref{Finitely many primes}.
\end{remark}

Another question, that is naturally related to Theorem
\ref{Prime Intersection <=> Condition K}, is to find conditions under which
every ideal of a Leavitt path algebra $L_{K}(E)$ is an intersection of maximal
ideals. For finite graphs $E$, more specifically when $E^{0}$ is finite, we
have a complete, easily derivable, answer.

\begin{proposition}
\label{Finite intersection of maximals} Let $E$ be a graph with $E^{0}$
finite. Then the following are equivalent for $L:=L_{K}(E)$:
\end{proposition}

\begin{enumerate}
\item[(a)] \textit{Every ideal of }$L$\textit{ is an intersection of maximal
ideals of }$L$\textit{;}

\item[(b)] $L=S_{1}\oplus\cdot\cdot\cdot\oplus S_{n}$\textit{ where }
$n>0$\textit{ is an integer, each }$S_{i}$\textit{ is a graded ideal which is
a simple ring with identity and }$\oplus$\textit{ is ring direct sum;}

\item[(c)] \textit{Every ideal of }$L$\textit{ is graded and is a ring direct
summand of }$L$\textit{;}

\item[(d)] \textit{The graph }$E$\textit{ satisfies Condition (K) and }$E^{0}
$\textit{ is the disjoint union of a finite number of hereditary saturated
subsets }$H_{i}$\textit{ each of which contains no non-empty proper hereditary
saturated subsets of vertices.}
\end{enumerate}

\begin{proof}
Assume (a). Since every maximal ideal is prime, Theorem
\ref{Prime Intersection <=> Condition K} implies that Condition (K) holds in
$E$ and so, by Theorem 6.16 of \cite{T}, every ideal of $L$ is graded. Since
$E^{0}$ is finite, we conclude, from the description of the graded ideals in
Theorem 5.17, \cite{T}, that $L$ contains only a finite number of ideals and,
in particular, finitely many maximal ideals. So we can write, by hypothesis,
$\{0\}=M_{1}\cap\cdot\cdot\cdot\cap M_{n}$ where the $M_{i}$ are all the
maximal ideals of $L$. Now apply the Chinese Remainder Theorem to conclude
that the map $a\mapsto(a+M_{1},\cdot\cdot\cdot,a+M_{n})$ is an epimorphism and
hence an isomorphism $\theta$ from $L$ to $L/M_{1}\oplus\cdot\cdot\cdot\oplus
L/M_{n}$. For each $i$, let $S_{i}$ be the (graded) ideal of $L$ isomorphic to
$L/M_{i}$ under the isomorphism $\theta$. Then $S_{i}$ is a simple ring with
identity and $L=\bigoplus\limits_{i=1}^{n}S_{i}$. This proves (b).

Assume (b). Let $A$ be a non-zero ideal of $L$. Then $A=(A\cap S_{1}
)\oplus\cdot\cdot\cdot\oplus(A\cap S_{n})$. Since the $S_{i}$ are all simple
rings, $A\cap S_{i}=0$ or $S_{i}$. Hence $A=S_{i_{1}}\oplus\cdot\cdot
\cdot\oplus S_{i_{k}}$, where $\{i_{1},\cdot\cdot\cdot,i_{k}\}\subseteq
\{1,\cdot\cdot\cdot,n\}$ and $A$ is a direct summand of $L$. Also $A$ is
clearly a graded ideal of $L$. This proves (c).

Assume (c). Since $L$ is a ring with identity $1$, $L$ contains maximal ideals
$M$ which are all direct summands of $L$ and whose complements $S$ will be
ideals containing no other non-zero ideals of $L$. Now the graded ideal $S$
contains local units, as it is isomorphic to a Leavitt path algebra, by
Theorem 6.1 of \cite{RT}. This implies that every ideal of $S$ is also an
ideal of $L$ and so the ideal $S$ will be a simple ring. Moreover, $S$ is
generated by a central idempotent, as it is a ring direct summand of a ring
with identity. By Zorn's Lemma choose a maximal family $\{S_{i}:i\in I\}$ of
distinct ideals $S_{i}$ of $L$ each of which is a simple ring. We claim that
$L=\sum\limits_{i\in I}S_{i}$. Otherwise, choose an ideal $M$ maximal with
respect to the property that $\sum\limits_{i\in I}S_{i}\subseteq M$, but
$1\notin M$. $M$ is clearly a maximal ideal of $L$. Since $M$ is a direct
summand, $L=M\oplus S$ and this case, $\{S\}\cup\{S_{i}:i\in I\}$ violates the
maximality of $\{S_{i}:i\in I\}$. Hence $L=\sum\limits_{i\in I}S_{i}$.
Actually, $L=\bigoplus\limits_{i\in I}S_{i}$ as each $S_{i}$ is generated by a
central idempotent. Since $L$ is a ring with identity, $I$ must be a finite
set. This proves (b).

Assume (b). For each $j$, let $M_{j}=\bigoplus\limits_{i\neq j,1\leq i\leq
n}S_{i}$ be a maximal ideal of $L$. Clearly $\bigcap\limits_{j=1}^{n}M_{j}=0$.
Moreover, if $A$ be any ideal of $L$, then $A$ is a direct sum of a
subcollection of the summands $S_{i}$ and then it is easy to see that
$A=\bigcap\limits_{A\subseteq M_{j}}M_{j}$. This proves (a).

We prove (b) $\Leftrightarrow$ (d). Assume (b). It is clear that each ideal
$A$ of $L$ is a direct sum of a subset of the graded ideals $S_{i}$ and hence
is graded. This implies that the graph $E$ satisfies Condition (K). Now for
each $i$, the graded ideal $S_{i}$ is a simple ring and so contains no
non-zero proper ideal of $L$. Hence $H_{i}=S_{i}\cap E^{0}$ is a hereditary
saturated set and contains no proper non-empty hereditary saturated subset of
vertices. Also $S_{i}S_{j}=0$ for all $i\neq j$ and this implies that the sets
$H_{i}$ are all pair-wise disjoint. Further $E^{0}=\bigcup\limits_{i=1}
^{n}H_{i}$. This proves (d).

Assume (d) so, for some $n>0$, $E^{0}=\bigcup\limits_{i=1}^{n}H_{i}$, where
the $H_{i}$ are pair-wise disjoint hereditary saturated subsets having no
proper non-empty hereditary saturated subsets of vertices in $E$. It is clear
that
\[
E^{0}\backslash H_{i}=\{u\in E^{0}:u\ngeqq v\text{ for any }v\in H_{i}\}.
\]
For each $i=1,\cdot\cdot\cdot,n$, let $E_{i}$ be the subgraph with
$(E_{i})^{0}=H_{i}$ and $(E_{i})^{1}=\{e\in E^{1}:s(e)\in H_{i}\}$. Clearly
each $E_{i}$ is a complete subgraph of $E$ satisfying Condition (K) and having
no proper non-empty hereditary saturated subsets of vertices. Moreover, the
graphs $E_{i}$ are all pair-wise disjoint. It then follows that $L\cong
\bigoplus\limits_{i=1}^{n}L_{K}(E_{i})$, where each $L_{K}(E_{i})$ is a simple
ring with identity (see \cite{AA}). This proves (b).
\end{proof}

\begin{remark}
\label{Intersection of maximals when E arbitrary} In the case when $E$ is an
arbitrary graph, we have the following (perhaps not a satisfactory) answer to
the above question: Every ideal of $L$ is an intersection of maximal ideals if
and only if, every ideal of $L$ is graded and for each ideal $A$ of $L$
(including the zero ideal), $L/A$ is a subdirect product of simple Leavitt
path algebras. To see this, note that if $A=\bigcap\limits_{i\in I}M_{i}$
where the $M_{i}$ are maximal ideals, then we get a homomorphism
$\theta:L\longrightarrow\prod\limits_{i\in I}L/M_{i}$ given by $x\longmapsto
(\cdot\cdot\cdot,x+M_{i},\cdot\cdot\cdot)$ with $\ker(\theta)=A$. Clearly
$\theta(L)$ maps onto $L/M_{i}$ under the coordinate projection $\eta
_{i}:\prod\limits_{i\in I}L/M_{i}\longrightarrow L/M_{i}$. \ This shows that
$L/A$ is a subdirect product of the simple rings $L/M_{i}$ each of which can
be realized as a Leavitt path algebra as each $M_{i}$ is a graded ideal of
$L$. Conversely, suppose $L/A\subseteq\prod\limits_{i\in I}L_{i}$ is a
subdirect product of simple rings $L_{i}$ and, for each $i\in I$, $\eta
_{i}:\prod\limits_{i\in I}L_{i}\longrightarrow L_{i}$ is the coordinate
projection. For each $i\in I$, let $A_{i}\supseteq A$ denote the ideal of $L$
such that $A_{i}=L\cap\ker(\eta_{i})$. Then it is easy to see that each
$A_{i}$ is a maximal ideal of $L$ and $A=\bigcap\limits_{i\in I}A_{i}$.
\end{remark}

\section{Prime factorization and powers of an ideal}

In this section, we consider the question of factorizing an ideal of a Leavitt
path algebra $L$ as a product of prime ideals. We first obtain a unique
factorization theorem for a graded ideal of $L$ as a product of prime ideals.
A perhaps interesting result is that if $I$ is a graded ideal and
$I=P_{1}\cdot\cdot\cdot P_{n}$ is a factorization of $I$ as an irredundant
product of prime ideals $P_{i}$, then necessarily all the ideals $P_{i}$ must
be graded ideals and moreover, $I=P_{1}\cap\ldots\cap P_{n}$. We also point
out a weaker factorization theorem for non-graded ideals as products of
primes. We end this section by showing that, given any non-graded ideal $I$ in
a Leavitt path algebra $L$, its powers $I^{n}$ $(n\geq1)$ are all non-graded
and distinct, but the intersection of its powers $\bigcap\limits_{n=1}
^{\infty}I^{n}$ is a graded ideal and is indeed the largest graded ideal of
$L$ contained in $I$. As a corollary, we obtain an analogue of Krull's theorem
for Leavitt path algebra (see \cite{ZS}): The intersection 
${\displaystyle\bigcap\limits_{n=1}^{\infty}} I^{n}=0$ for an ideal $I$ of $L$ 
if and only if $I$ contains no vertices of $E$.

We begin with a useful property of graded ideals.

\begin{lemma}
\label{Intersection = product} Let $E$ be an arbitrary graph and let $I$ be a
graded ideal of $L:=L_{K}(E)$. Then $I=P_{1}\cdot\cdot\cdot P_{n}$ is a
product of arbitrary ideals $P_{i}$ if and only if $I=P_{1}\cap\ldots\cap
P_{n}$.
\end{lemma}

\begin{proof}
Suppose $I=P_{1}\cap\ldots\cap P_{n}$. Clearly $P_{1}\cdot\cdot\cdot
P_{n}\subseteq P_{1}\cap\ldots\cap P_{n}=I$. To prove the reverse inclusion,
note that, by (Theorem 6.1, \cite{RT}), the graded ideal $I$ is isomorphic to
a Leavitt path algebra of a suitable graph and so it contains local units. Let
$a\in P_{1}\cap\ldots\cap P_{n}$. Then there is a local unit $u=u^{2}\in I$
such that $ua=a=au$. Since, for each $i$, $u\in P_{i}$, multiplying $a$ by $u$
on the right $(n-1)$ times, we obtain $a=au\ldots u\in P_{1}\cdot\cdot\cdot
P_{n}$. Hence $I=P_{1}\cdot\cdot\cdot P_{n}$.

Conversely, suppose $I=P_{1}\cdot\cdot\cdot P_{n}$. If $I\neq P_{1}\cap
\ldots\cap P_{n}$, then, in $L/I$, $(P_{1}\cap\ldots\cap P_{n})/I$ is a
non-zero nilpotent ideal whose $n-$th power is zero. This is a contradiction
since $L/I$ is isomorphic to a Leavitt path algebra, as $I$ is a graded ideal.
Hence $I=P_{1}\cap\ldots\cap P_{n}$.
\end{proof}

As \ the operation $\cap$ is commutative, we obtain from the preceding Lemma
the following corollary.

\begin{corollary}
\label{Permuted product} If a graded ideal $I$ of $L_{K}(E)$ is a product of
ideals $I=P_{1}\cdot\cdot\cdot P_{n}$, then $I$ is equal to any permuted
product of these ideals, that is, $I=P_{\sigma(1)}\cdot\cdot\cdot
P_{\sigma(n)}$ where $\sigma$ is a permutation of the set $\{1,\ldots,n\}$.
\end{corollary}

We now are ready to prove a uniqueness theorem in factorizing a graded ideal
of a Leavitt path algebra as a product of prime ideals. Recall that
$I=P_{1}\cdot\cdot\cdot P_{n}$ is an \textbf{irredundant product }of the
ideals $P_{i}$, if $I$ is not the product of a proper subset of this set of
$n$ ideals $P_{i}$.

\begin{theorem}
\label{Graded Prime factorization}Let $E$ be an arbitrary graph and let $I$ be
a graded ideal of $L:=L_{K}(E)$.
\end{theorem}

\begin{enumerate}
\item[(a)] \textit{If }$I=P_{1}\cdot\cdot\cdot P_{m}$\textit{ is an
irredundant product of prime ideals }$P_{i}$\textit{, then all the ideals
}$P_{i}$\textit{ are graded and }$I=P_{1}\cap\ldots\cap P_{m}$\textit{.}

\item[(b)] \textit{If}
\[
I=P_{1}\cdot\cdot\cdot P_{m}=Q_{1}\cdot\cdot\cdot Q_{n}
\]
\textit{are two irredundant products of prime ideals }$P_{i}$\textit{ and
}$Q_{j}$\textit{, then }$m=n$\textit{ and }
\[
\{P_{1},\ldots,P_{m}\}=\{Q_{1},\ldots,Q_{m}\}\text{\textit{.}}
\]

\end{enumerate}

\begin{proof}
(a) Now, by Lemma \ref{Intersection = product}, $I=P_{1}\cap\ldots\cap P_{m}$
which is an irredundant intersection as $P_{1}\cdot\cdot\cdot P_{m}$ is an
irredundant product. Note that the graded ideal $I\subseteq gr(P_{i})$ for all
$i=1,\ldots,n$, where $gr(P_{i})$ denotes the largest graded ideal contained
in $P_{i}$ (see Lemma 3.6, \cite{R-1}). Now
\[
I\subseteq gr(P_{1})\cap\ldots\cap gr(P_{m})\subseteq P_{1}\cap\ldots\cap
P_{m}=I\text{.}%
\]
So we conclude that $I=gr(P_{1})\cap\ldots\cap gr(P_{m})$. A priori, it is not
clear whether $I=gr(P_{1})\cap\ldots\cap gr(P_{m})$ is an irredundant
intersection. Suppose $\{gr(P_{i_{1}}),\ldots,gr(P_{i_{k}})\}$ is a
subcollection of the ideals $gr(P_{i})$ such that $I=gr(P_{i_{1}})\cap
\ldots\cap gr(P_{i_{k}})$ is an irredundant intersection. Note that each
$gr(P_{i_{r}})$, for $r=1,\ldots,k$, is a prime ideal, as each $P_{i_{r}}$ is
a prime ideal (Lemma 3.8, \cite{R-1}). Then, by Proposition \ref{Uniqueness},
$k=m$ and $\{gr(P_{i_{1}}),\ldots,gr(P_{i_{k}})\}=\{P_{1},\ldots,P_{m}\}$. By
irredundancy, each $P_{i}=gr(P_{i})$ and hence is a graded ideal.

(b) If $I=P_{1}\cdot\cdot\cdot P_{m}=Q_{1}\cdot\cdot\cdot Q_{n}$ are two
irredundant products of prime ideals $P_{i}$ and $Q_{j}$, then, by Lemma
\ref{Intersection = product},
\[
I=P_{1}\cap\ldots\cap P_{m}=Q_{1}\cap\ldots\cap Q_{n}
\]
are two irredundant intersections of prime ideals and so, by Proposition
\ref{Uniqueness}, $m=n$ and $\{P_{1},\ldots,P_{m}\}=\{Q_{1},\ldots,Q_{m}
\}$\textit{.}
\end{proof}

As noted earlier, for graded ideals $A$ and $B$, the property that $A\cap
B=AB$ came in handy in proving the uniqueness theorem. This property does not
always hold for non-graded ideals. For an easy example consider Example
\ref{Laurent => No Prime Intersection}. Observe that every non-zero ideal of
$L_{K}(E)\cong K[x,x^{-1}]$ is non-graded. Let $A=\left\langle (v+c)^{2}
\right\rangle $ and $B=\left\langle v-c^{2}\right\rangle $. Then $A\cap B\neq
AB$ as $A\cap B=\left\langle (v+c)(v-c^{2})\right\rangle $ while
$AB=\left\langle (v+c)^{2}(v-c^{2})\right\rangle $.

For ideals which are not necessarily graded, we next prove a weaker version of
a uniqueness theorem. For convenience, we call a product of ideals $P_{1}
\cdot\cdot\cdot P_{m}$ \textbf{tight }if $P_{i}\nsubseteq P_{j}$ for all
$i\neq j$. Note that a tight product of prime ideals is necessarily an
irredundant product.

\begin{proposition}
\label{Non-graded factorization} Let $E$ be an arbitrary graph. Suppose
\[
A=P_{1}\cdot\cdot\cdot P_{m}=Q_{1}\cdot\cdot\cdot Q_{n}
\]
are two representations of an ideal $A$ of $L_{K}(E)$ as tight products of
prime ideals $P_{i}$ and $Q_{j}$. Then $m=n$ and $\{P_{1},\ldots
,P_{m}\}=\{Q_{1},\ldots,Q_{m}\}$\textit{.}
\end{proposition}

\begin{proof}
Now the prime ideal $P_{1}\supseteq Q_{1}\cdot\cdot\cdot Q_{n}$ and so
$P_{1}\supseteq Q_{i_{1}}$ for some $i_{1}$. By a similar argument, $Q_{i_{1}
}\supseteq P_{j}$ for some $j$. Then $P_{1}\supseteq Q_{i_{1}}\supseteq P_{j}$
and since the product is tight, $P_{1}=P_{j}$. So $P_{1}=Q_{i_{1}}$. Next
start with $P_{2}$ and proceed as before to conclude that $P_{2}=Q_{i_{2}}\neq
Q_{i_{1}}$. Proceeding like this we conclude that $\{P_{1},\ldots
,P_{m}\}\subseteq\{Q_{1},\ldots,Q_{n}\}$. Reversing the role and starting with
the $Q$'s and proceeding as before, we get $\{Q_{1},\ldots,Q_{n}
\}\subseteq\{P_{1},\ldots,P_{m}\}$. Thus $m=n$ and $\{P_{1},\ldots
,P_{m}\}=\{Q_{1},\ldots,Q_{m}\}$.
\end{proof}

Next we consider the powers of an ideal $I$. We begin with the following
useful Lemma.

\begin{lemma}
\label{Power of a non-graded ideal} Suppose $E$ is an arbitrary graph and
$L:=L_{K}(E)$. Let $c$ be a cycle in $E$ with no exits based at a vertex $v$
and let $B=\left\langle p(c)\right\rangle $ be the ideal generated by $p(c)$
in $L$, where $p(x)=1+k_{1}x+\cdot\cdot\cdot+k_{n}x^{n}\in K[x]$. Then
\end{lemma}

\begin{enumerate}
\item[(a)] $vB^{m}v=(vBv)^{m}$\textit{, for any }$m>0$\textit{;}

\item[(b)] $B^{m}\neq B^{n}$\textit{ for all }$0<m<n$\textit{.}
\end{enumerate}

\begin{proof}
(a) Clearly $(vBv)^{m}\subseteq vB^{m}v$. We show that $vB^{m}v\subseteq
(vBv)^{m}$. Now a typical element of $vB^{m}v$ is a $K$-linear sum of finitely
many terms each of which, being a product of $m$ elements of $B$ followed by
multiplication by $v$ on both sides, is of the form
\begin{equation}
v[\alpha_{1}\beta_{1}^{\ast}p(c)\gamma_{1}\delta_{1}^{\ast}][\alpha_{2}
\beta_{2}^{\ast}p(c)\gamma_{2}\delta_{2}^{\ast}]\cdot\cdot\cdot\lbrack
\alpha_{m}\beta_{m}^{\ast}p(c)\gamma_{m}\delta_{m}^{\ast}]v\tag{$\ast$}
\end{equation}
where the $\alpha_{i},\beta_{i},\gamma_{i},\delta_{i}$ are all paths in $E$.
Now $(\ast)$ can be rewritten as{\small
\[
\lbrack v\alpha_{1}\beta_{1}^{\ast}p(c)v][v\gamma_{1}\delta_{1}^{\ast}
\alpha_{2}\beta_{2}^{\ast}p(c)v][v\gamma_{2}\delta_{2}^{\ast}\alpha_{3}
\beta_{3}^{\ast}p(c)v]\cdot\cdot\cdot p(c)v][v\gamma_{m-1}\delta_{m-1}^{\ast
}\alpha_{m}\beta_{m}^{\ast}p(c)\gamma_{m}\delta_{m}^{\ast}v]
\]
}which is clearly a product of $m$ elements of $vBv$ and hence belongs to
$(vBv)^{m}$. Thus $vB^{m}v\subseteq(vBv)^{m}$ and we are done.

(b) Since $c$ has no exits, $vLv\overset{\theta}{\cong}K[x,x^{-1}]$ where
$\theta$ maps $v$ to $1$, $c$ to $x$ and $c^{\ast}$ to $x^{-1}$. As
$vp(c)v=p(c)$, $vBv=B\cap vLv$ contains $p(c)$ and is the ideal generated by
$p(c)$ in $vLv$. Thus $vBv$ is isomorphic to the ideal $\left\langle
p(x)\right\rangle $ in $K[x,x^{-1}]$ under the map $\theta$. If $B^{m}=B^{n}$
for some $0<m<n$, then $vB^{m}v=vB^{n}v$. By (a), we then get $(vBv)^{m}
=(vBv)^{n}$ and this implies that, in the principal ideal domain $K[x,x^{-1}
]$, $\left\langle p(x)^{m}\right\rangle =\left\langle p(x)^{n}\right\rangle $
for $0<m<n$, a contradiction. Hence $B^{m}\neq B^{n}$ for all $0<m<n$.
\end{proof}

Observe that if $I$ is an ideal of $L_{K}(E)$ such that $I^{n}$ is a graded
ideal for some $n>1$, then $I$ must be a graded ideal. Indeed $I=I^{n}$.
Because, if $I\neq I^{n}$, then $I/I^{n}$ becomes a non-zero nilpotent ideal
in $L_{K}(E)/I^{n}$, a contradiction as $L_{K}(E)/I^{n}$ is isomorphic to a
Leavitt path algebra and its Jacobson radical is zero.

It then follows from the preceding observation that if $I$ is a non-graded
ideal of $L_{K}(E)$, then for any integer $n>0$, $I^{n}$ must also be a
non-graded ideal. But, as Theorem \ref{Intersection of powers} below shows,
the intersection of its powers $\bigcap\limits_{n=1}^{\infty}I^{n}$ must be a
graded ideal.

\begin{theorem}
\label{Intersection of powers}Let $I$ be a non-graded ideal of a Leavitt path
algebra $L$ of an arbitrary graph $E$. If $H=I\cap E^{0}$ and $S=\{v\in
B_{H}:v^{H}\in I\}$, then $\bigcap\limits_{n=1}^{\infty}I^{n}=I(H,S)$, the
largest graded ideal contained in $I$.
\end{theorem}

\begin{proof}
Now $L/I(H,S)\cong L_{K}(E\backslash(H,S))$. Identifying $\bar{I}=I/I(H,S)$
with its isomorphic image in $L_{K}(E\backslash(H,S))$, $\bar{I}$ is an ideal
containing no vertices and so, by (Proposition 2, \cite{R-2}), $\bar{I}$ is
generated by an orthogonal set $\{p_{j}(c_{j}):j\in Y\}$, where $Y$ is an
index set and, for each $j\in Y$, $c_{j}$ is a cycle without exits in
$E\backslash(H,S)$ and $p_{j}(x)\in K[x]$. For each $j\in Y$, let $A_{j}$ be
the ideal generated by the vertices on $c_{j}$. It was shown in (Proposition
3.5(ii), \cite{AAPS}) that the ideal sum $\sum\limits_{j\in Y}A_{j}%
=\bigoplus\limits_{j\in Y}A_{j}$. Since $p_{j}(c_{j})\in A_{j}$, $\bar
{I}=\bigoplus\limits_{j\in Y}B_{j}$ where $B_{j}$ is the ideal generated by
$p_{j}(c_{j})$. By Proposition 3.5(iii) of \cite{AAPS}, each $A_{j}$ is
isomorphic to $M_{\Lambda_{j}}(K[x,x^{-1}])$, where $\Lambda_{j}$ is a
suitable index set representing the number of paths that end at $c_{j}$ but do
not contain all the edges of $c_{j}$. So $B_{j}$ is isomorphic to an ideal
$N_{j}$ of $M_{\Lambda_{j}}(K[x,x^{-1}])$. As $M_{\Lambda_{j}}(K[x,x^{-1}])$
is Morita equivalent to $K[x,x^{-1}]$ (\cite{A}, \cite{AM}), there is a
lattice isomorphism $\phi:\mathbf{L\longrightarrow L}^{\prime}$, where
$\mathbf{L,L}^{\prime}$ are the lattices of ideals of $M_{\Lambda_{j}
}(K[x,x^{-1}])$ and $K[x,x^{-1}]$, respectively \cite{AF}. Now, for a fixed
$j$, $B_{j}^{m}\neq B_{j}^{n}$ for all $0<m<n$, by Lemma
\ref{Power of a non-graded ideal}. Hence the corresponding ideal $N_{j}$ also
satisfies $N_{j}^{m}\neq N_{j}^{n}$ for all $0<m<n$. So we get an infinite
descending chain of ideals
\begin{equation}
\phi(N_{j})\supset\phi(N_{j}^{2})\supset\ldots\supset\phi(N_{j}^{n}
)\supset\ldots\tag{$\ast\ast$}
\end{equation}
in $K[x,x^{-1}]$. Let $N=\bigcap\limits_{n=1}^{\infty}\phi(N_{j}^{n})$. If
$N\neq0$, then $K[x,x^{-1}]/N$ satisfies the descending chain condition, as
$K[x,x^{-1}]$ is principal ideal domain\ (see e.g. Theorem 32, Ch. IV-15,
\cite{ZS}). This is a contradiction since the chain $(\ast\ast)$ induces an
infinite descending chain of ideals in $K[x,x^{-1}]/N$. Hence $N=\bigcap
\limits_{n=1}^{\infty}\phi(N_{j}^{n})=0$. Then Lemma \ref{Lattice Iso} implies
that $\bigcap\limits_{n=1}^{\infty}B_{j}^{n}=0$ for each $j$. Since $\bar{I}$
is a (ring) direct sum of the ideals $B_{j}$, $\bigcap\limits_{n=1}^{\infty
}(\overline{I})^{n}=0$. This means that $\bigcap\limits_{n=1}^{\infty}
I^{n}=I(H,S)$. As noted in Lemma 3.6 of \cite{R-1}, $I(H,S)$ is the largest
graded ideal of $L$ contained in $I$. This completes the proof.
\end{proof}

\begin{remark}
In the proof of Theorem \ref{Intersection of powers}, observe that the direct
sum of ideals $\bar{I}=\bigoplus\limits_{j\in Y}B_{j}$ satisfies $(\bar
{I})^{m}\neq(\bar{I})^{n}$ for all positive integers $m\neq n$, as each
$B_{j}$ satisfies the same property. Clearly, the same holds for the ideal
$I$. Thus every non-graded ideal $I$ of a Leavitt path algebra satisfies
$I^{m}\neq I^{n}$ for all integers $0<m<n$. Moreover, by the statements
preceding Theorem \ref{Intersection of powers}, each $I^{n}$ must be a
non-graded ideal.
\end{remark}

W. Krull showed that if $I$ is an ideal of a noetherian integral domain $R$,
then $\bigcap\limits_{n=1}^{\infty}I^{n}=0$ and more generally, if $R$ is a
commutative noetherian ring, then $\bigcap\limits_{n=1}^{\infty}I^{n}=0$ if
and only if $1-x$ is not a zero divisor for all $x\in I$ (see Theorem 12,
Section 7 in \cite{ZS}). From Theorem \ref{Intersection of powers} and its
proof, using the fact that a non-zero graded ideal always contains a vertex,
one can easily obtain the following analogue of Krull's theorem for Leavitt
path algebras.

\begin{corollary}
\label{Krull's theorem} Suppose $I$ is an ideal of a Leavitt path algebra
$L_{K}(E)$ of an arbitrary graph $E$. Then the intersection 
${\displaystyle\bigcap\limits_{n=1}^{\infty}}I^{n}=0$ if and only if 
$I$ contains no vertices of $E$.
\end{corollary}

\textbf{Acknowledgement:} We thank Professor Astrid an Huef for raising the
question of intersections of primitive ideals in a Leavitt path algebra and
Professor Pere Ara for helpful preliminary discussion on this question during
the CIMPA conference. The first and the second authors also thank Duzce
University Distance Learning Center (UZEM) for using UZEM facilities while
conducting this research.

\end{document}